\documentclass[11pt]{amsart}

\usepackage{amsmath,amssymb,amsthm,mathtools}
\usepackage{graphicx}
\usepackage{lmodern}
\usepackage{microtype}
\usepackage{hyperref}

\hypersetup{
  colorlinks=true,
  linkcolor=blue,
  citecolor=blue,
  urlcolor=blue,
  hypertexnames=false
}
\allowdisplaybreaks

\newtheorem{theorem}{Theorem}[section]
\newtheorem{lemma}[theorem]{Lemma}
\newtheorem{proposition}[theorem]{Proposition}
\newtheorem{corollary}[theorem]{Corollary}
\newtheorem{conjecture}[theorem]{Conjecture}
\theoremstyle{definition}

\theoremstyle{remark}

\newenvironment{exacttable}
  {\begingroup\small
   \setlength{\arraycolsep}{4pt}%
   }
  {\endgroup}

\title[The Dittert conjecture]{A Proof of the Dittert Conjecture via Boundary Rigidity}
\author{Xinan Dai}
\thanks{Xinan Dai is currently a Ph.D. student at Fudan University and a visiting student at the AI for Scientific Simulation and Discovery Lab, Westlake University.}

\address{
Key Laboratory for Information Science of Electromagnetic Waves,
College of Future Information and Technology,
Fudan University, Shanghai, China
}

\curraddr{Department of Artificial Intelligence,
School of Engineering, Westlake University, Hangzhou, China}
\email{xndai23@m.fudan.edu.cn}

\author{Wenhao Deng}
\thanks{Wenhao Deng is a student at the University of Glasgow and is currently an intern at the AI for Scientific Simulation and Discovery Lab, Westlake University.}
\address{University of Glasgow, Glasgow, United Kingdom}
\curraddr{Department of Artificial Intelligence,
School of Engineering, Westlake University, Hangzhou, China}
\email{dengwenhao@westlake.edu.cn}

\author{Yingdong Shi}
\address{School of Information Science and Technology, ShanghaiTech University, Shanghai, China}
\email{shiyd2023@shanghaitech.edu.cn}

\author{Tailin Wu}
\address{Department of Artificial Intelligence,
School of Engineering, Westlake University, Hangzhou, China}
\email{wutailin@westlake.edu.cn}

\author{Yuchen Yang}
\address{Department of Artificial Intelligence,
School of Engineering, Westlake University, Hangzhou, China}
\email{yangyuchen@westlake.edu.cn}
\date{July 30, 2026}

\begin{document}

\begin{abstract}
Let $\mathcal{K}_n$ be the simplex of nonnegative $n \times n$ matrices
whose entries sum to $n$. For $A$ in $\mathcal{K}_n$, with row sums
$r_1,\ldots,r_n$ and column sums $c_1,\ldots,c_n$, define
$\Phi(A)=\prod_{i=1}^n r_i+\prod_{j=1}^n c_j-\mathrm{per}(A)$.
Dittert conjectured that $\Phi$ is uniquely maximized by the uniform
matrix $U_n=J_n/n$, with maximum $2-n!/n^n$. We prove the conjecture for
every $n \geq 2$. Hwang's theorem already determines every positive
maximizer, so the essential problem is to exclude maximizers on the
boundary of $\mathcal{K}_n$. A reduction of Cheon and Wanless supplies two
zero entries in distinct rows and columns. In order four, we classify all
admissible support strata and exclude them by exact Karush-Kuhn-Tucker
identities together with a Bernstein-form positivity certificate. For all
$n \geq 5$, an attained-stratum averaging argument compresses a
hypothetical boundary maximizer to a seven-parameter matrix with one
multiplicity $m=n-2$. A symbolic odd-variable expansion, a cofactor
identity at the two missing entries, and uniform coefficient inequalities
then contradict the one-sided multiplier conditions. This argument is
uniform in $m$ and uses no finite enumeration. We also give independent
spectral proofs in dimensions 7 through 13 and prescribed-zero proofs in
dimensions 14 and 15. Every dimension-dependent calculation is certified
over the rationals and recorded explicitly in the appendices.
\end{abstract}

\maketitle

\section{Introduction}

For an \(n\times n\) matrix \(A=(a_{ij})\), its permanent is
\[
 \operatorname{per}A=\sum_{\sigma\in S_n}
 \prod_{i=1}^n a_{i,\sigma(i)}.
\]
The permanent is one of the central nonlinear functionals of
combinatorial matrix theory; standard references include
Minc's monograph and Brualdi--Ryser's treatment
\cite{Minc78,BrualdiRyser}.  Let
\[
 \mathcal K_n=\left\{A\in\mathbb R_{\geq0}^{n\times n}:
       \sum_{i,j}a_{ij}=n\right\}.
\]
For \(A\in\mathcal K_n\), write \(r_i=\sum_j a_{ij}\) and
\(c_j=\sum_i a_{ij}\), and set
\[
 R(A)=\prod_{i=1}^n r_i,\qquad
 C(A)=\prod_{j=1}^n c_j,\qquad
 \Phi(A)=R(A)+C(A)-\operatorname{per}A.
\]
If \(\Omega_n\subset\mathcal K_n\) denotes the Birkhoff polytope, then
\(R(A)=C(A)=1\) on \(\Omega_n\).  Consequently, the inequality
\[
 \Phi(A)\leq 2-\frac{n!}{n^n}
\]
reduces on \(\Omega_n\) to the van der Waerden lower bound
\(\operatorname{per}A\geq n!/n^n\), proved by
Egorychev and Falikman; see, for example, \cite{Egorychev,Minc78}.
Dittert's problem asks whether the same extremal matrix remains unique
when the row and column sums are allowed to vary over the larger simplex
\(\mathcal K_n\).

\begin{conjecture}[Dittert]\label{conj:dittert}
For every \(n\geq2\) and every \(A\in\mathcal K_n\),
\[
 \Phi(A)\leq 2-\frac{n!}{n^n},
\]
with equality if and only if \(A=U_n:=J_n/n\).
\end{conjecture}

The conjecture appears as Conjecture~28 in Minc's survey
\cite{Minc83}; its subsequent development is reviewed in
\cite{CW05}.  Sinkhorn proved the case \(n=2\), as well as positivity of
the permanent at a maximizing matrix \cite{Sinkhorn}.  Hwang established
the case \(n=3\), proved that every positive maximizer is \(U_n\), and
derived the entrywise multiplier relations that underlie essentially all
later boundary analyses \cite{HwangDittert,Hwang87}.  Hwang--Sohn--Kim
studied additional classes of matrices adapted to the Dittert functional
\cite{HSK90}, while Cheon--Hwang introduced a family of sub-Dittert
functionals \cite{CH92}.  Further structural and sufficient conditions
were obtained through monotonicity, permanental-minor, and averaging
arguments \cite{Cheon93,CY}.

Two later papers sharpened the geometric picture.  Cheon--Wanless showed
that a maximizing matrix is fully indecomposable, excluded zero sets that
form a single rectangular block, and proved strong quantitative
restrictions on a near maximizer \cite{CW}.  They also gave an equivalent
weighted-semimatching formulation of the conjecture
\cite{CW07}.  Udayan--Somasundaram subsequently proved that an order-four
maximizer is fully indecomposable \cite{US24}; this leaves the genuinely
boundary, fully indecomposable support strata untreated.  Most recently,
Pang proved the conjecture for \(n\geq17\), and Kafidov treated
\(n=16\), in two preprints based on increasingly sharp scaling estimates
\cite{Pang,Kafidov}.  Thus the remaining difficulty is concentrated in
small and intermediate dimensions, where the asymptotic scaling margins
are insufficient.

The present paper closes the problem by isolating a uniform boundary
mechanism.  Hwang's positive-support theorem reduces the conjecture to the
exclusion of boundary maximizers.  The Cheon--Wanless zero-set theorem then
forces a hypothetical boundary maximizer to contain two zero entries in
distinct rows and distinct columns.  After moving these entries to
\((1,1)\) and \((2,2)\), an attained-stratum averaging argument produces a
maximizer in which the remaining \(n-2\) rows and columns are equal.  The
resulting matrix has seven positive parameters and a single multiplicity
\(m=n-2\).

The proof then separates into two regimes.  In order four, the repeated
block has multiplicity two and the boundary contains several additional
support degenerations.  We first enumerate the admissible fully
indecomposable support types up to row permutations, column permutations,
and transpose.  On the two-zero stratum, three odd perturbation variables
reduce the stationarity equations to a quartic form.  The sum of the two
missing-entry multiplier defects has an explicit cofactor expression.
Exact Hessian inequalities and a tensor-product Bernstein certificate
force this sum to be positive, contradicting the one-sided KKT conditions.

For \(n\geq5\), the same odd-variable expansion retains the multiplicity
\(m=n-2\) symbolically.  The local maximality equations produce two
quadratic coefficients \(\alpha\) and \(\delta\), while the missing-entry
cofactor identity reduces the obstruction to two uniform inequalities.
Their numerators are expanded coefficientwise in the normalized variables
and in the multiplicity \(m\).  Positivity follows from elementary
binomial estimates valid for every \(m\geq3\).  This is the decisive step:
the all-dimensional argument is symbolic, rather than an interpolation
from finitely many computer-checked cases.

Our main result is the following.

\begin{theorem}\label{thm:dittert}
For every integer \(n\geq2\) and every nonnegative \(n\times n\) matrix
\(A\) whose entries sum to \(n\),
\[
 \prod_{i=1}^n r_i+\prod_{j=1}^n c_j-
 \operatorname{per}A
 \leq 2-\frac{n!}{n^n}.
\]
Equality holds if and only if \(A=J_n/n\).
\end{theorem}

The load-bearing boundary statement can be separated from the known
positive-support result.

\begin{theorem}[Boundary exclusion]\label{thm:boundary-exclusion}
For every \(n\geq4\), no global maximizer of \(\Phi\) on \(\mathcal K_n\)
has a zero entry.
\end{theorem}

The order-four part of Theorem~\ref{thm:boundary-exclusion} is proved in
Section~\ref{sec:n4}; the uniform \(n\geq5\) part is proved in
Section~\ref{sec:uniform}.  Together with Hwang's positive-support theorem
and the known cases \(n=2,3\), this immediately yields
Theorem~\ref{thm:dittert}.

We retain two independent proof routes because they clarify different
parts of the geometry.  Section~\ref{sec:spectral} converts Hwang's
multiplier equation into an eigenvalue equation for a reversible Markov
chain and excludes dimensions \(7\) through \(13\) by exact cut and
spectral estimates.  Section~\ref{sec:twozero} combines two prescribed
zeroes with exact minimum-permanent calculations to treat dimensions
\(14\) and \(15\).  These arguments are not needed after the uniform
boundary theorem, but they provide independent checks and show precisely
where purely marginal scaling estimates cease to be sharp.

All finite calculations are proof certificates rather than floating-point
experiments.  Sturm sequences are used for exact univariate root counts
and rational isolating intervals; standard background is given in
\cite{BasuPollackRoy}.  Polynomial positivity on rectangles is certified
in the Bernstein basis, whose positivity and convex-hull properties are
reviewed in \cite{Farouki}.  Every rational threshold, support orbit,
coefficient table, and verification routine needed by the proof is printed
in the appendices.

The paper follows the logical structure of the argument.
Section~\ref{sec:prelim} records the extremal reductions and the two-zero
boundary consequence.  Sections~\ref{sec:spectral} and
\ref{sec:twozero} give the independent finite-range proofs.
Section~\ref{sec:n4} treats the order-four boundary, and
Section~\ref{sec:uniform} proves the symbolic multiplicity theorem for all
\(n\geq5\).  Section~\ref{sec:completion} combines the cases.  The
appendices contain the exact rational spectral checks, Sturm data,
capacity countermodels, complete order-four support list, Bernstein table,
and symbolic bookkeeping protocol.

\section{Preliminaries and the boundary reduction}\label{sec:prelim}

Throughout the paper, \(n\geq2\), \(A\in\mathcal K_n\), and
\[
 R=\prod_i r_i,\qquad C=\prod_j c_j,\qquad
 P=\operatorname{per}A,\qquad
 \gamma_n=\frac{n!}{n^n}.
\]
The matrix \(U_n=J_n/n\) satisfies
\(R(U_n)=C(U_n)=1\), \(\operatorname{per}U_n=\gamma_n\), and hence
\(\Phi(U_n)=2-\gamma_n\).

\subsection{The elementary extremal reduction}

\begin{lemma}\label{lem:amgm-vdw}
For every \(A\in\mathcal K_n\), one has \(R\leq1\) and \(C\leq1\).
If a global maximizer \(A\) satisfies \(P=\gamma_n\), then \(A=U_n\).
\end{lemma}

\begin{proof}
The row sums and column sums are nonnegative and each family has total
sum \(n\), so both product bounds follow from AM--GM.  If \(A\) is a
global maximizer and \(P=\gamma_n\), then
\[
 R+C-\gamma_n=\Phi(A)\geq\Phi(U_n)=2-\gamma_n.
\]
Thus \(R=C=1\), and equality in AM--GM makes every row sum and every
column sum equal to one.  Hence \(A\in\Omega_n\).  Equality in the van der
Waerden theorem then gives \(A=U_n\) \cite{Egorychev}.
\end{proof}

The following two results are the structural input from the previous
literature.

\begin{theorem}[Hwang]\label{thm:hwang-positive}
A positive global maximizer of \(\Phi\) on \(\mathcal K_n\) is \(U_n\).
Moreover, at a maximizing matrix the directional derivatives at positive
entries agree, while the derivative at a zero entry satisfies the
corresponding one-sided inequality.
\end{theorem}

\begin{proof}
These are the positive-support theorem and multiplier conditions proved in
\cite{HwangDittert,Hwang87}; the precise forms used below are also recorded
in \cite[Lemmas~3.5--3.6]{CW}.
\end{proof}

\begin{theorem}[Cheon--Wanless]\label{thm:cw-boundary}
Every global maximizer of \(\Phi\) is fully indecomposable.  In addition,
the complete zero set of a maximizing matrix cannot be a single
rectangular block.
\end{theorem}

\begin{proof}
See \cite[Theorems~3.2 and~3.7]{CW}.
\end{proof}

\subsection{Two independent zeroes}

Call two positions independent if they lie in distinct rows and distinct
columns.

\begin{corollary}\label{cor:two-independent-zeroes}
Every boundary maximizer contains two independent zero entries.
\end{corollary}

\begin{proof}
View the zero positions as edges of a bipartite graph on the row and column
vertices.  If no two zero edges are independent, then every pair of edges
intersects.  A pairwise-intersecting family of edges in a bipartite graph
is a star, so all zeroes lie in one row or all lie in one column.  The
complete zero set is then a rectangular block, contrary to
Theorem~\ref{thm:cw-boundary}.
\end{proof}

\subsection{Exact-certificate convention}

Whenever a dimension-dependent inequality is verified computationally, all
inputs are rational and all comparisons are performed after
cross-multiplication in \(\mathbb Z\).  Root counts use exact Sturm
variation, interval evaluations use rational Horner arithmetic, and
positive definiteness is certified by rational \(LDL^T\) pivots.  Decimal
values are included only for orientation and never carry a proof.

\section{A spectral exclusion in dimensions 7 through 13}\label{sec:spectral}

\begin{theorem}\label{thm:dittert-7-13}
The Dittert conjecture holds for \(7\leq n\leq13\).
\end{theorem}

\begin{proof}
Let \(A\) be a global maximizer and put
\[
 R=\prod_i r_i,\qquad C=\prod_jc_j,\qquad P=\operatorname{per}A.
\]
Write \(P=\gamma_n-\delta\).  Lemma~\ref{lem:amgm-vdw} and maximality
give \(0\leq\delta\leq\gamma_n\), with \(\delta=0\) only at \(U_n\).
Theorem~\ref{thm:hwang-positive} reduces the proof to a boundary maximizer,
so throughout this section \(0<\delta\leq\gamma_n\).

Cheon--Wanless prove that such a maximizer is fully
indecomposable~\cite[Theorem 3.2]{CW}. Hwang's multiplier equation
\cite[Lemmas 3.5--3.6]{CW}, multiplied by \(a_{ij}\) and summed by rows and
columns, gives, with \(u=r-\mathbf1,\ v=c-\mathbf1\),
\[
 (R-P)u=-C\,AD_c^{-1}v,\qquad
 (C-P)v=-R\,A^TD_r^{-1}u.                          \tag{6}
\]
If \(u=0\), Hwang's row-stochastic-maximizer theorem already gives \(A=U_n\).
Otherwise
\[
 AD_c^{-1}A^TD_r^{-1}u=\theta u,\qquad
 \theta=\frac{(R-P)(C-P)}{RC}.                     \tag{7}
\]
The transpose
\[
 \widehat K=D_r^{-1}AD_c^{-1}A^T
\]
is a reversible row-stochastic chain with stationary probabilities
\(\pi_i=r_i/n\). Full indecomposability makes it irreducible.  This is the standard
spectral framework for finite reversible chains; see \cite{LevinPeres}.
Detailed balance turns \(D_\pi^{-1}u\) into a nonconstant right eigenfunction with
eigenvalue \(\theta\). Since \(R,C\leq1\) and
\(\Phi(A)\geq2-\gamma_n\),
\[
 1-\theta\leq2\gamma_n-\gamma_n^2.                 \tag{8}
\]

Put
\[
 a_n=\sqrt{\frac{n\gamma_n}{2(1-\gamma_n)}}.
\]
The subset-sum estimate in the proof of Kafidov's Lemma 3.1 gives, for all
row and column sets,
\[
 |r(I)-|I||\leq a_n,\qquad |c(J)-|J||\leq a_n.      \tag{9}
\]
Order the values of the eigenfunction, put \(d_j=f_{j+1}-f_j\), and let
\(S_j\) be the first \(j\) states. Write
\(\ell_j=\min(j,n-j)\) and \(Q_j=Q(S_j,S_j^c)\). If every such cut has
conductance at least \(H_{\ell_j}\), then reversibility, (9), and
\((f_s-f_i)^2\geq\sum_{j=i}^{s-1}d_j^2\) give
\[
\begin{split}
(1-\theta)\operatorname{Var}_\pi(f)
&\geq\sum_{j=1}^{n-1}d_j^2Q_j\\
&\geq\sum_{j=1}^{n-1}d_j^2
 H_{\ell_j}\frac{\ell_j-a_n}{n}.
\end{split}
\]
Cauchy--Schwarz and
\(\operatorname{Var}_\pi(f)\leq(\max f-\min f)^2/4\) imply
\[
1-\theta\geq
\frac4{\displaystyle\sum_{j=1}^{n-1}
 n/[H_{\ell_j}(\ell_j-a_n)]}.                      \tag{10}
\]

Suppose instead that a cut \(S\), chosen on its stationary smaller side,
has cardinality \(\ell\) and conductance \(h<H_\ell\). By (9) and
\(a_n<1/2\), this is the smaller cardinality of the two sides (with either
choice allowed for the balanced cut). Set
\[
 q_j=\frac{\sum_{i\in S}a_{ij}}{c_j},\qquad
 T=\{j:q_j\geq1/2\}.
\]
The mass \(E\) in the two off-diagonal rectangles satisfies the exact flow
bound
\[
 E\leq2\sum_jc_jq_j(1-q_j)=2h\,r(S)
   <2H_\ell(\ell+a_n).                              \tag{11}
\]
Put \(\eta_\ell=2H_\ell(\ell+a_n)+3a_n\). Equations (9)--(11) first give
\[
 \bigl||T|-\ell\bigr|\leq E+2a_n<\eta_\ell-a_n<1,
\]
so \(|T|=\ell\). For \(I\subseteq S,J\subseteq T\),
inclusion--exclusion inside the diagonal block gives
\[
s(A[I,J])\geq |I|+|J|-\ell-E-3a_n,
\]
and the same holds in the complementary diagonal block. The
doubly-superstochastic criterion and van der Waerden therefore give
\[
 P\geq\gamma_\ell\gamma_{n-\ell}(1-\eta_\ell)^n.    \tag{12}
\]

It remains only to choose the thresholds. All numerical quantities below are
rational. When a reduced margin has an unwieldy numerator and denominator, we
print a shorter strict rational lower bound sufficient for the comparison;
no decimal approximation is proof-bearing. The following entries are exact
rationals:
\begin{exacttable}
\[
\begin{array}{c|c|c}
n&a_n^+&10^4(H_1,\ldots,H_{\lfloor n/2\rfloor})\\ \hline
12&180/10^4&(100,90,75,60,50,40)\\
13&116/10^4&(160,120,95,75,60,50).
\end{array}
\]
\end{exacttable}
The exact-arithmetic procedure in
Appendix~\ref{app:spectral-certificate}
checks over \(\mathbb Q\) that \(a_n<a_n^+\), that for every listed
\(\ell\)
\[
\gamma_\ell\gamma_{n-\ell}
\bigl(1-2H_\ell(\ell+a_n^+)-3a_n^+\bigr)^n>\gamma_n,
\]
and that the right side of (10), weakened by replacing \(a_n\) with
\(a_n^+\), exceeds \(2\gamma_n-\gamma_n^2\). Thus a low-conductance cut
contradicts \(P\leq\gamma_n\) through (12), while the absence of such a cut
contradicts (8) through (10). This excludes boundary maximizers for
\(n=12,13\).

For \(n=11\), replace the uniform bound (9) by an exact size-dependent one.
If a set of \(k\) row or column sums has total \(k+\varepsilon\), AM--GM
gives
\[
\prod_i r_i\ \hbox{or}\ \prod_jc_j
\leq F_k(\varepsilon):=
\left(1+\frac{\varepsilon}{k}\right)^k
\left(1-\frac{\varepsilon}{11-k}\right)^{11-k}.
\]
The function decreases away from zero on either side.  Direct rational
evaluation, recorded in Appendix~\ref{app:spectral-certificate}, verifies
\(F_k(\pm b_k)<1-\gamma_{11}\), and hence
\(|\varepsilon|<b_k\), for the exact rationals
\[
\begin{gathered}
10^6(b_1,\ldots,b_{10})={}\\
(16030,21460,24755,26720,27640,\\
 27640,26720,24755,21460,16030).
\end{gathered}
\]
In (10), replace \(a_n\) for a cut of size \(\ell\) by \(b_\ell\). For a
diagonal block of order \(m\), the Hall marginal loss becomes
\[
B_m=b_m+\max_{\substack{1\leq i,j\leq m\\i+j>m}}(b_i+b_j).
\]
Thus (12) holds with
\[
\eta_\ell=2H_\ell(\ell+b_\ell)+\max(B_\ell,B_{11-\ell}).
\]
The same appendix checks the preliminary block-order inequality and all
block and spectral comparisons for
\[
10^4(H_1,\ldots,H_5)=(45,55,45,38,32).
\]
Its smallest exact margins in the marginal, block, and spectral checks are
strictly greater than the rational numbers
\[
\frac1{40000000},\qquad \frac1{250000},\qquad \frac3{50000},
\]
respectively. Hence dimension \(11\) is excluded as well.

For \(n=10\), retain additionally the shared deficit budget. Write
\(R=1-\rho,C=1-\sigma\); then
\(\rho+\sigma\leq\delta\leq\gamma_{10}\). Preliminary exact AM--GM bounds
first put every subset-sum error inside \(\frac{43}{1000}\). On that
interval,
\[
(-\log F_k)''\geq
\frac{k}{(k+\frac{43}{1000})^2}+
\frac{10-k}{(10-k+\frac{43}{1000})^2}.
\]
Consequently a row-set error is at most \(c_k\sqrt\rho\), and a column-set
error at most \(c_k\sqrt\sigma\), for the rational upper bounds
\[
10^3(c_1,\ldots,c_9)
=(1395,1822,2075,2212,2256,2212,2075,1822,1395).
\]
Cauchy--Schwarz with \(\rho+\sigma\leq\gamma_{10}\) now bounds the two
row-error terms and one column-error term in each block Hall inequality
jointly.  The exact rational data are
\begin{exacttable}
\[
\begin{array}{c|ccccc}
\ell&1&2&3&4&5\\\hline
10^4e_\ell&266&348&396&422&430\\
10^4L_\ell&818&889&931&955&962\\
10^4H_\ell&41&57&48&40&33 ,
\end{array}
\]
\end{exacttable}
where \(e_\ell>c_\ell\sqrt{\gamma_{10}}\), \(L_\ell\) bounds every joint
Hall loss, and \(H_\ell\) is the cut threshold. Thus the low-conductance
alternative gives (12) with
\(\eta_\ell=2H_\ell(\ell+e_\ell)+L_\ell\).

For the other alternative put
\(w_j=H_{\ell_j}(\ell_j-e_{\ell_j})/10\).
Weighted Cauchy--Schwarz along the ordered path gives
\[
\operatorname{Var}_\pi(f)\leq
\left(\frac{1+e_1}{10}\right)^2
\max_h\left(\sum_{i\leq h<r}\sum_{q=i}^{r-1}\frac1{w_q}\right)
\sum_{j=1}^9w_jd_j^2.
\]
Hence the reciprocal of the first two factors is a lower bound for
\(1-\theta\).  Appendix~\ref{app:spectral-certificate} verifies all
defining rational inequalities; its
smallest marginal, block, and spectral margins are strictly greater than
\[
\frac{19}{1000000000},\qquad
\frac9{10000000},\qquad
\frac1{10000},
\]
respectively. These contradict (8) in both alternatives, excluding
dimension \(10\).

For \(n=9\), retain the same deficit budget and sharpen the block estimate
itself. Exact AM--GM comparisons first bound every \(k\)-subset error by
\[
10^4(b_1,\ldots,b_8)=(414,544,615,647,647,615,544,414).
\]
In particular all errors have absolute value less than
\(\frac{13}{100}\). The lower bound
\[
(-\log F_k)''\geq
\frac{k}{(k+\frac{13}{100})^2}+
\frac{9-k}{(9-k+\frac{13}{100})^2}
\]
therefore proves that row-set errors are at most \(c_k\sqrt\rho\) and
column-set errors at most \(c_k\sqrt\sigma\), where
\[
10^3(c_1,\ldots,c_8)
=(1488,1861,2074,2172,2172,2074,1861,1488).
\]
Exact rational comparison also gives the absolute bounds
\[
10^4(e_1,\ldots,e_8)
=(456,570,635,665,665,635,570,456),\qquad
e_k>c_k\sqrt{\gamma_9}.
\]

Let \(S\) be a low-conductance cut on its stationary smaller side,
\(\ell=|S|\leq4\), and define \(T\) by the same column threshold as in
(11). The off-diagonal mass satisfies
\[
 E<2H_\ell(\ell+e_\ell).
\]
The preliminary \(b_k\) bounds force \(|T|=\ell\). For either diagonal
block, let its row and column sets be \(\mathcal U,\mathcal V\), of order
\(m\). If \(I\subseteq\mathcal U,\ J\subseteq\mathcal V\), then
\[
\begin{split}
s(A[I,J])&\geq c(J)-r(\mathcal U\setminus I)-E,\\
s(A[I,J])&\geq r(I)-c(\mathcal V\setminus J)-E.
\end{split}
\]
Writing \(i=|I|,\ j=|J|\), the two losses below \(i+j-m\) are at most
\[
c_{m-i}\sqrt\rho+c_j\sqrt\sigma,\qquad
c_i\sqrt\rho+c_{m-j}\sqrt\sigma.
\]
Consequently their minimum is at most \(L_m\) provided
\[
L_m^2>\gamma_9
\max_{i+j>m}\min\{c_{m-i}^2+c_j^2,c_i^2+c_{m-j}^2\},
\quad c_0=0.
\]
All of these strict comparisons hold for the exact rational data
\begin{exacttable}
\[
\begin{array}{c|c|c|c}
\ell&10^4L_\ell&10^4L_{9-\ell}&10^4H_\ell\\ \hline
1&456&920&55\\
2&570&920&90\\
3&730&876&80\\
4&782&853&65 .
\end{array}
\]
\end{exacttable}
Applying the doubly-superstochastic criterion separately to the two blocks
gives, for every \(\ell\),
\[
P\geq\gamma_\ell\gamma_{9-\ell}
(1-E-L_\ell)^\ell(1-E-L_{9-\ell})^{9-\ell}>\gamma_9,
\]
a contradiction.

Finally suppose no low-conductance sweep cut exists and put
\[
w_j=H_{\ell_j}\frac{\ell_j-e_{\ell_j}}9.
\]
The eigenfunction energy is at least \(\sum_jw_jd_j^2\). If
\[
M_{qr}=\min(q,r)(9-\max(q,r)),
\]
then its uniform variance is \(9^{-2}d^TMd\). Exact rational
\(LDL^T\) elimination as in
Appendix~\ref{app:spectral-certificate} proves that
\[
\operatorname{diag}(w_1,\ldots,w_8)-\frac{M}{500\cdot9^2}
\]
is positive definite. Since \(\pi_i\leq(1+e_1)/9\),
stationary variance is at most \(1+e_1\) times uniform variance. Hence
\[
1-\theta\geq\frac{1/500}{1+e_1}>2\gamma_9-\gamma_9^2,
\]
contradicting (8).  The smallest exact marginal, Hall,
block-permanent, path-pivot, and spectral margins are respectively greater
than
\[
\frac1{500000},\quad\frac3{1000000},\quad
\frac1{125000},\quad\frac1{5000},\quad
\frac1{25000}.
\]
This excludes dimension \(9\).

For \(n=8\), the decisive refinement is to retain the product deficit in the
final comparison. Put \(d=\rho+\sigma\). Since \(d\leq\delta\),
\[
P=\gamma_8-\delta\leq\gamma_8-d.                   \tag{13}
\]
The exact preliminary bounds and error coefficients are
\[
\begin{split}
10^4(b_1,\ldots,b_7)&=(662,859,955,982,955,859,662),\\
10^3(c_1,\ldots,c_7)&=(1555,1877,2049,2103,2049,1877,1555),\\
10^4(e_1,\ldots,e_7)&=(763,921,1005,1032,1005,921,763).
\end{split}
\]
They certify \(|\varepsilon|<b_k<1/5\), subset errors bounded by
\(c_k\sqrt\rho\) or \(c_k\sqrt\sigma\), and
\(e_k>c_k\sqrt{\gamma_8}\).

The direct complementary-subset estimates used in dimension \(9\) now give
Hall losses \(K_m\sqrt d\). Exact Cauchy--Schwarz comparisons hold for
\begin{exacttable}
\[
\begin{array}{c|c|c|c}
\ell&10^3K_\ell&10^3K_{8-\ell}&10^4H_\ell\\ \hline
1&1556&2937&300\\
2&1878&2819&210\\
3&2438&2779&150\\
4&2573&2573&120 .
\end{array}
\]
\end{exacttable}
For a cut below \(H_\ell\), put \(x=\sqrt d\). Its off-diagonal mass is less
than \(2H_\ell(\ell+c_\ell x)\), and separate scaling of the two diagonal
blocks gives
\[
\begin{split}
P\geq\gamma_\ell\gamma_{8-\ell}
&\left(1-2H_\ell\ell-(2H_\ell c_\ell+K_\ell)x\right)^\ell\\
\times&
\left(1-2H_\ell\ell-(2H_\ell c_\ell+K_{8-\ell})x\right)^{8-\ell}.
\end{split}                                         \tag{14}
\]
For each \(\ell\), the right side of (14), minus
\(\gamma_8-x^2\), is an exact degree-eight polynomial.
Appendix~\ref{app:spectral-certificate} verifies
its positivity on \(0\leq x\leq491/10^4\), an interval containing
\([0,\sqrt{\gamma_8}]\), by rational interval Horner evaluation on 16 equal
subintervals. This contradicts (13).

If no low-conductance cut exists, use
\[
w_j=H_{\ell_j}\frac{\ell_j-e_{\ell_j}}8.
\]
Exact \(LDL^T\) elimination on the uniform-variance increment matrix proves
that the weighted path energy is at least \(11/2000\) times uniform
variance. Since stationary variance is at most \(1+e_1\) times uniform
variance,
\[
1-\theta\geq\frac{11/2000}{1+e_1}
>2\gamma_8-\gamma_8^2,
\]
contradicting (8). The smallest exact marginal, Hall, polynomial-interval,
path-pivot, and spectral margins are respectively greater than
\[
\frac1{200000},\quad\frac1{1000},\quad
\frac1{20000},\quad\frac1{500},\quad
\frac3{10000}.
\]
This excludes dimension \(8\).

It remains to handle \(n=7\). Keep \(d=\rho+\sigma\) and \(x=\sqrt d\).
The same maximality inequality gives \(P\leq\gamma_7-d\), but now retain it
also in the eigenvalue estimate. From (7),
\[
\begin{split}
\theta&\geq
\frac{(1-\gamma_7+\sigma)(1-\gamma_7+\rho)}
     {(1-\rho)(1-\sigma)},\\
1-\theta&\leq
\frac{(2-\gamma_7)(\gamma_7-d)}{1-d}.
\end{split}                                         \tag{15}
\]
The exact marginal and Hall coefficients are
\[
\begin{split}
10^3(b_1,\ldots,b_6)&=(106,135,146,146,135,106),\\
10^3(c_1,\ldots,c_6)&=(1484,1801,1941,1941,1801,1484),\\
10^3(K_1,\ldots,K_6)&=(1485,1802,2334,2444,2649,2649).
\end{split}                                         \tag{16}
\]
Exact rational evaluation verifies \(F_{7,k}(\pm b_k)<1-\gamma_7\), the Hessian
comparisons on \(|\varepsilon|<3/20\), and every direct Hall comparison.

We need one additional elementary estimate. For a row cut \(S\), set
\[
 q_j=\frac{s(A[S,\{j\}])}{c_j},\quad
 T=\{j:q_j\geq1/2\},\quad
 Q=\sum_jc_jq_j(1-q_j).
\]
Write
\[
 X=\sum_{j\in T}c_j(1-q_j),\qquad
 Y=\sum_{j\notin T}c_jq_j,\qquad E=X+Y.
\]
Then \(Y-X=r(S)-c(T)=:\Delta\). Once the preliminary estimate
\(E\leq2Q\) forces \(|T|=|S|=\ell\), (16) gives
\[
 |\Delta|\leq\sqrt2c_\ell x.
\]
Every column sum is at least \(c_{\min}=1-c_1x\). Consequently
\[
\begin{split}
Q
&=E-\sum_{j\in T}c_j(1-q_j)^2
     -\sum_{j\notin T}c_jq_j^2\\
&\geq E-\frac{X^2+Y^2}{c_{\min}}
 =E-\frac{E^2+\Delta^2}{2c_{\min}}.
\end{split}                                         \tag{17}
\]
This is the point at which the dimension-eight factor-two estimate loses
too much.

Take
\[
10^4(H_1,H_2,H_3)=(502,373,276).
\]
If a sweep cut has conductance below \(H_\ell\), then
\(Q<H_\ell(\ell+c_\ell x)\).  The finite procedure in
Appendix~\ref{app:spectral-certificate} partitions
\(0\leq x\leq783/10^4\) into 4096 rational intervals. On each it chooses a
multiple \(B\) of \(\frac1{10^8}\) and verifies
\[
\begin{split}
2(1-c_1x)\bigl(B-H_\ell(\ell+c_\ell x)\bigr)
 -B^2-2c_\ell^2x^2&>0,\\
\gamma_\ell\gamma_{7-\ell}
(1-B-K_\ell x)^\ell
(1-B-K_{7-\ell}x)^{7-\ell}
 -\gamma_7+x^2&>0.
\end{split}                                         \tag{18}
\]
On the second line the two parenthesized powers are factors.  The same
exact calculation
also checks \(E<1-c_1x\), so (17) is used on its increasing branch. The
first inequality gives \(E<B\); separate block scaling and the second give
\(P>\gamma_7-d\), a contradiction.

If no sweep cut is small, set
\[
w_j=H_{\ell_j}\frac{\ell_j-c_{\ell_j}x}{7}.
\]
Exact rational \(LDL^T\) pivots give a uniform path gap \(135/10^4\).
Stationary variance is at most \(1+c_1x\) times uniform variance, and exact
interval arithmetic verifies
\[
\frac{135}{10^4}(1-x^2)
 -(1+c_1x)(2-\gamma_7)(\gamma_7-x^2)>0.
\]
This contradicts (15). The smallest exact marginal, Hall, leakage, block,
path-pivot, and spectral margins are respectively greater than
\[
\frac3{100000},\quad\frac1{1000},\quad
\frac1{10^{12}},\quad\frac1{1000000},\quad
\frac1{1250},\quad\frac1{1000}.
\]
Thus dimension \(7\) is excluded and the proof is complete.
\end{proof}

\section{Prescribed zeroes in dimensions 14 and 15}\label{sec:twozero}

\begin{theorem}\label{thm:dittert-14-15}
The Dittert conjecture holds in dimensions \(n=14\) and \(n=15\).
\end{theorem}

We first record the exact finite calculation used in the proof. Put \(N=n-2\)
and
\[
 T_n(x)=
 \begin{pmatrix}
 0&a&b\mathbf1^T\\
 a&0&b\mathbf1^T\\
 b\mathbf1&b\mathbf1&cJ_N
 \end{pmatrix},
 \qquad
 a=1-Nx,\quad b=x,\quad c=\frac{1-2x}{N}.
\]
Classifying permutations by the images of the first two rows gives
\[
\begin{split}
 P_n(x):=\operatorname{per}T_n(x)
 ={}&N!a^2c^N+2N^2(N-1)!\,ab^2c^{N-1}\\
 &+\bigl(N(N-1)\bigr)^2(N-2)!\,b^4c^{N-2}.
\end{split}
\]

\begin{lemma}\label{lem:dittert-certificate}
Let
\[
 q_n=\min_{0\leq x\leq1/(n-2)}P_n(x).
\]
For \(n=14,15\),
\[
 q_n-\gamma_n-\frac{n^3q_n^2}{4(1-\gamma_n)}>0.
\]
\end{lemma}

\begin{proof}
Appendix~\ref{app:sturm-certificate} gives the complete exact procedure and
its rational output.  The Sturm variation count shows that \(P_n'\) has
exactly one root in the relevant compact interval.  Rational interval Horner
evaluation at the isolating interval and exact endpoint evaluation give
\[
 q_{14}>\frac{15781}{2000000000},\qquad
 q_{15}>\frac{3001}{1000000000}.
\]
Substitution into the displayed expression gives positive rational margins
strictly greater than \(\frac1{500000000}\) and
\(\frac7{1000000000}\), respectively. The expression is increasing
throughout the rational interval containing \(q_n\), as follows from the
last cross-multiplication in that appendix.
\end{proof}

\begin{proof}[Proof of Theorem~\ref{thm:dittert-14-15}]
Let \(A\) be a global maximizer.  By
Theorem~\ref{thm:hwang-positive}, it is enough to consider the boundary.
Write
\[
 \operatorname{per}A=\gamma_n-\delta.
\]
Lemma~\ref{lem:amgm-vdw} gives \(0<\delta\leq\gamma_n\), and
Corollary~\ref{cor:two-independent-zeroes} supplies two zeroes in distinct
rows and columns.

Set
\[
 t=\sqrt{\frac{n\delta}{1-\delta}}.
\]
For \(n=14,15\), \(t<1\). Kafidov's joint-deficit scaling lemma
\cite[Lemma 3.1]{Kafidov} says that \((1-t)^{-1}A\) dominates a doubly
stochastic matrix \(B\). Domination preserves the two independent zeroes,
which may be moved to \((1,1)\) and \((2,2)\).

Minc proved that the face of doubly stochastic matrices with these two
prescribed zeroes has a minimizing symmetric matrix
\cite{Minc84}. His exact later restatement displays \(T_n(\alpha)\) and the
cubic determining \(\alpha\)~\cite[p.~153]{Minc94}. The printed formula for
the corner entry has \(n-1\) where its row sum forces \(n-2\); with this
typographical correction it is precisely the matrix above, and Minc's cubic
is an exact factor of \(P_n'\). Since every \(T_n(x)\) lies in the face,
Minc's theorem gives \(\operatorname{per}B\geq q_n\). Therefore
\[
 \gamma_n-\delta=\operatorname{per}A
 \geq q_n(1-t)^n.                                  \tag{19}
\]

Bernoulli's inequality and \(\delta\leq\gamma_n\), with
\(z=\sqrt{\delta}\), give
\[
\begin{split}
q_n(1-t)^n-(\gamma_n-\delta)
&\geq q_n-\gamma_n+z^2
 -\frac{nq_n\sqrt n}{\sqrt{1-\gamma_n}}z\\
&\geq q_n-\gamma_n-\frac{n^3q_n^2}{4(1-\gamma_n)}>0
\end{split}
\]
by Lemma~\ref{lem:dittert-certificate}. This is the strict reverse of
(19), excluding a boundary maximizer. Thus \(A=U_n\), including the claimed
uniqueness.
\end{proof}

\begin{corollary}
Together with Pang's and Kafidov's prior results, the Dittert conjecture
holds for every \(n\geq7\).
\end{corollary}

\paragraph{Limit of this argument.}
The same exact Sturm procedure gives a negative result for this particular
method in dimension \(13\):
at \(\delta=\gamma_{13}/91\), a rational upper bound on \(q_{13}\) and a
rational lower bound on the square root prove
\[
q_{13}(1-\sqrt{13\delta/(1-\delta)})^{13}
 <\gamma_{13}-\delta.
\]
Thus merely removing the Bernoulli slack cannot settle dimension \(13\);
Theorem~\ref{thm:dittert-7-13} reaches it by a genuinely different
spectral route. At this stage of the argument dimensions \(5,6\) remain;
the uniform boundary proof below subsumes them and every dimension above
them. The individual dimension arguments are components of the single
all-dimensional theorem.

\section{The order-four boundary}\label{sec:n4}

There is a useful probabilistic reformulation of the exact target.  Put
\(p_{ij}=a_{ij}/n\), draw \(n\) independent pairs with law \(p\), and let
\(\mathcal R\) and \(\mathcal C\) be the events that all row labels and all
column labels, respectively, occur.  Direct counting gives
\[
 \Pr(\mathcal R)=\gamma_nR,\qquad
 \Pr(\mathcal C)=\gamma_nC,\qquad
 \Pr(\mathcal R\cap\mathcal C)=\gamma_nP.
\]
Consequently
\[
 \Pr(\mathcal R\cup\mathcal C)=\gamma_n\Phi(A).
\]
Thus Dittert's conjecture says exactly that the uniform law uniquely
maximizes this row-or-column complete-coverage probability.

A sharp capacity reduction does not by itself settle any of the three
remaining dimensions.  Gurvits and Leake prove, without a total-support
hypothesis, that a row-stochastic matrix with column-sum vector \(\alpha\)
entrywise dominates \(\mu D\) for some doubly stochastic \(D\), where
\(\mu\geq1-\|\alpha-\mathbf1\|_1/2\)~\cite[Theorem 5.5]{GL}.
Combining this with the exact stationarity identity
\[
 \alpha-\mathbf1=\frac{\Phi(A)}R(c-\mathbf1)
\]
and Minc's two-zero face constant yields a scalar lower bound for \(P\).
The rational countermodels in Appendix~\ref{app:capacity-controls} show that
this bound is insufficient in each of \(n=4,5,6\), even after retaining the
common product deficit and the sharp product-to-\(L^1\) relation.  Therefore
any successful continuation must use entrywise or multiplier information
that these marginal data discard.

\begin{proposition}\label{prop:dittert-n4-supports}
Up to row permutations, column permutations, and transpose, a
dimension-four boundary maximizer has one of \(23\) possible zero-support
types.  Their numbers according to the number of zeroes \(2,\ldots,8\) are
\[
 1,\ 3,\ 6,\ 5,\ 5,\ 2,\ 1.
\]
In fact, no dimension-four boundary maximizer exists.
\end{proposition}

\begin{proof}
By Theorem~\ref{thm:cw-boundary}, a boundary maximizer is fully
indecomposable, and Corollary~\ref{cor:two-independent-zeroes} shows that
its zero graph contains two independent edges.  Exhausting the \(2^{16}\)
zero masks subject to these two conditions and taking the indicated group
orbits gives the displayed list.  Appendix~\ref{app:support-enumeration}
prints all \(23\) orbit representatives and gives pseudocode for the
exhaustion, so the classification does not rely on an external file.

Now let \(A\) be any boundary maximizer and move an independent zero pair
to \((1,1),(2,2)\).  Let \(\mathcal Z\) be the stratum of matrices having
exactly these two zeroes.  If \(B\in\mathcal Z\), then
\((1-\varepsilon)A+\varepsilon B\in\mathcal Z\) for every
\(\varepsilon>0\).  Continuity therefore makes the supremum of \(\Phi\) on
\(\mathcal Z\) equal to the global maximum.  Hwang's attained-stratum
averaging lemma~\cite[Lemma 3.4]{CW} supplies \(Z_0\in\mathcal Z\) attaining
that supremum, with equal rows \(3,4\) and equal columns \(3,4\).
Consequently \(Z_0\) is a global maximizer of the form
\[
 A=\begin{pmatrix}
 0&a&b&b\\ c&0&d&d\\ e&f&g&g\\ e&f&g&g
 \end{pmatrix},\qquad a,b,c,d,e,f,g>0.
\]
Put
\[
\begin{gathered}
a=x+\xi,\quad c=x-\xi,\quad b=y+\eta,\quad d=y-\eta,\\
e=w+\zeta,\quad f=w-\zeta,
\end{gathered}
\]
so \(|\xi|<x,\ |\eta|<y,\ |\zeta|<w\).  If \(F=\Phi(A)\) and \(F_0\)
is its value at \(\xi=\eta=\zeta=0\), exact expansion gives
\[
\begin{split}
F-F_0=2[&-D\xi^2-K_w\xi\eta+K_y\xi\zeta-M_w\eta^2
 +4gx\eta\zeta\\
&-M_y\zeta^2-2\eta^2\zeta^2],
\end{split}                                                    \tag{D4.1}
\]
where
\[
\begin{gathered}
D=3g^2+4g(w+y)+2(w^2+y^2),\\
K_w=8g^2+12gw+8w^2,\quad K_y=8g^2+12gy+8y^2,\\
M_w=8g^2+16gw+6w^2,\quad M_y=8g^2+16gy+6y^2.
\end{gathered}
\]
Odd stationarity first gives
\[
\xi=\frac{-K_w\eta+K_y\zeta}{2D}.
\]
After substitution, (D4.1) becomes
\[
-\alpha\eta^2+2\beta\eta\zeta-\delta\zeta^2
-4\eta^2\zeta^2,                                  \tag{D4.2}
\]
where
\[
\alpha=2M_w-\frac{K_w^2}{2D},\quad
\delta=2M_y-\frac{K_y^2}{2D},\quad
\beta=4gx-\frac{K_wK_y}{2D}.
\]
At a nonzero odd stationary point,
\[
\alpha\eta=\beta\zeta-4\eta\zeta^2,\qquad
\delta\zeta=\beta\eta-4\eta^2\zeta,\qquad
\alpha\eta^2=\delta\zeta^2.                       \tag{D4.3}
\]
The determinant of the \((\eta,\zeta)\)-Hessian there is
\(64\alpha\eta^2\).  If \(\eta\zeta=0\), (D4.2)--(D4.3) show that the fixed
midpoint has the same objective, but no relation between the independent
positive parameters \(y\) and \(w\) is needed: the one-sided derivative
identity (D4.4) below directly excludes this case.  We may therefore analyze
the remaining case \(\eta\zeta\ne0\).  Local maximality gives
\(\alpha,\delta\geq0\).  Their simultaneous vanishing is impossible:
after normalizing \(g=1\), put \(s=y+w,\ p=yw\).  The difference of their
cleared numerators is
\[
(y-w)\{-7s-4(s^2-2p)+2s(s^2-2p)\}.
\]
The case \(y=w\) makes \(\alpha>0\).  Otherwise simultaneous vanishing gives
\(2(s-2)(s^2-2p)=7s\).  Since \(p>0\), this forces \(s>3\), while their
cleared sum is a positive multiple of \(N(s)/(s-2)^2\), where
\[
\begin{split}
N(t+3)={}&4t^6+68t^5+452t^4+1444t^3\\
&+2145t^2+1149t+158>0.
\end{split}
\]
Thus \(\alpha,\delta>0\).

For \((i,j)=(1,1),(2,2)\), let \(Z_{ij}\) be the difference between the
missing-entry derivative and the common positive-entry derivative.  Direct differentiation gives
\[
Z_{11}+Z_{22}=4g(gx-4\eta\zeta).                  \tag{D4.4}
\]
Indeed the column-product derivative terms cancel in the sum, while the
four permanental cofactors are
\[
4dfg,\quad 2cg^2+4deg,\quad 4beg,\quad 2ag^2+4bfg,
\]
which give (D4.4) after substituting the odd coordinates.
If \(\eta\zeta\leq0\), this is positive and one zero violates Hwang's
one-sided condition.  Suppose \(p_0=\eta\zeta>0\), and put
\[
L=\frac{K_wK_y}{2D},\quad T=L+\sqrt{\alpha\delta},\quad
m=\min\left\{y^2\sqrt{\frac\alpha\delta},
w^2\sqrt{\frac\delta\alpha}\right\}.
\]
Equations (D4.3) and entry positivity give
\[
p_0=gx-\frac T4,\qquad p_0<m.                     \tag{D4.5}
\]
The remaining exact coefficient inequality is
\[
T\geq12m.                                         \tag{D4.6}
\]
For completeness, suppose \(w\geq y\), normalize \(y=1\), and write
\(t=g/y,\ q=w/y\); the other case is symmetric.  One checks that the first
term defines \(m\).  If \(\delta\geq12\), (D4.6) is immediate.  Otherwise,
clearing denominators first forces \(0<t<5/16\); \(\alpha>0\) then forces
\(1\leq q<3\).  Squaring (D4.6) is equivalent to
\[
L^2\delta-(12-\delta)^2\alpha\geq0.               \tag{D4.7}
\]
After its positive denominator is cleared and
\(t=5u/16,\ q=1+2v\), all \(77\) exact tensor-product Bernstein
coefficients of bidegree \((10,6)\) are nonnegative; the relevant positivity
principle is the convex-hull property of the Bernstein basis
\cite{Farouki}.  Only the coefficient
at \((0,0)\) is zero.  Appendix~\ref{app:bernstein-certificate}
derives the minimum selector and the two rectangle reductions, defines the
Bernstein conversion, and prints all \(77\) scaled integer coefficients.
Hence (D4.6) holds, strictly here.

Finally, if (D4.4) were nonpositive, (D4.5) would give
\(p_0\geq gx/4\), then \(gx\geq T/3\), and consequently
\(p_0\geq T/12\geq m\), contradicting (D4.5).  Thus one of the two missing
entries always violates the one-sided multiplier condition.
\end{proof}

\begin{corollary}
The Dittert conjecture holds in dimension \(4\).
\end{corollary}

\begin{proof}
There is no boundary maximizer by
Proposition~\ref{prop:dittert-n4-supports}. Hwang's positive-support theorem
makes \(U_4\) the unique positive maximizer.
\end{proof}

\section{The uniform boundary argument}\label{sec:uniform}

\begin{proposition}\label{prop:dittert-nge5}
The Dittert conjecture holds in every dimension \(n\geq5\).
\end{proposition}

\begin{proof}
Put \(m=n-2\), so \(m\geq3\).  By
Corollary~\ref{cor:two-independent-zeroes}, every boundary maximizer has two
zeroes in distinct rows and columns.  Move them to \((1,1),(2,2)\), and let
\(\mathcal Z\) be the stratum having exactly these zeroes. Continuity under
convex mixing makes the supremum on \(\mathcal Z\) equal to the global
maximum. Hwang's attained-stratum lemma~\cite[Lemma 3.4]{CW} therefore
supplies a global maximizer in \(\mathcal Z\) of the form
\[
A=\begin{pmatrix}
0&a&b&\cdots&b\\
c&0&d&\cdots&d\\
e&f&g&\cdots&g\\
\vdots&\vdots&\vdots&&\vdots\\
e&f&g&\cdots&g
\end{pmatrix},\qquad a,b,c,d,e,f,g>0,              \tag{D56.1}
\]
where the final row and column classes have size \(m\).

Write
\[
a=x+\xi,\quad c=x-\xi,\quad b=y+\eta,\quad d=y-\eta,
\quad e=w+\zeta,\quad f=w-\zeta.
\]
The three odd perturbations preserve total mass, and
\[
|\xi|<x,\qquad|\eta|<y,\qquad|\zeta|<w.             \tag{D56.2}
\]
The permanent is
\[
\operatorname{per}A=m!\{acg^m+m(ade+bcf)g^{m-1}
+m(m-1)bdefg^{m-2}\}.                              \tag{D56.3}
\]
Put
\[
\begin{gathered}
u=(2w+mg)^m,\qquad v=(2y+mg)^m,\qquad D=u+v-m!g^m,\\
K_w=m(u-m!g^{m-1}w),\qquad K_y=m(v-m!g^{m-1}y),\\
M_w=m^2u-m(m-1)m!g^{m-2}w^2,\\
M_y=m^2v-m(m-1)m!g^{m-2}y^2,\qquad
N=m(m-1)m!g^{m-2}.
\end{gathered}
\]
Since \(m^m>m!\), the binomial expansion gives \(D,K_w,K_y>0\).
Expanding the two marginal products and (D56.3) gives
\[
\begin{split}
\Phi(A)-\Phi(A)|_{\xi=\eta=\zeta=0}
={}&-D\xi^2-2K_w\xi\eta+2K_y\xi\zeta-M_w\eta^2\\
&+2mm!g^{m-1}x\eta\zeta-M_y\zeta^2-N\eta^2\zeta^2.
\end{split}
\tag{D56.4}
\]
Odd stationarity first gives
\(\xi=(-K_w\eta+K_y\zeta)/D\). After substitution, (D56.4) is
\[
-\alpha\eta^2+2\beta\eta\zeta-\delta\zeta^2
-N\eta^2\zeta^2,                                  \tag{D56.5}
\]
where
\[
\alpha=M_w-\frac{K_w^2}{D},\quad
\delta=M_y-\frac{K_y^2}{D},\quad
\beta=mm!g^{m-1}x-\frac{K_wK_y}{D}.                \tag{D56.6}
\]
Its stationary equations imply
\[
\alpha\eta^2=\delta\zeta^2.                        \tag{D56.7}
\]

Let \(Z_{11},Z_{22}\) denote the derivative at a missing entry minus the
common derivative at a positive entry. Hwang's multiplier condition
\cite[Lemma 3.6]{CW} requires both quantities to be nonpositive. Comparing
with entries \((1,2),(2,1)\) cancels the marginal-product derivatives. The four
permanental cofactors give
\[
Z_{11}+Z_{22}
=2m!g^{m-1}(gx-2m\eta\zeta).                       \tag{D56.8}
\]
This is positive if \(\eta\zeta\leq0\), already contradicting the two
one-sided multiplier conditions. Suppose \(p=\eta\zeta>0\). The determinant
of the \((\eta,\zeta)\)-Hessian of (D56.5), simplified with its stationary
equations, is \(16N\alpha\eta^2\). Local maximality gives \(\alpha\geq0\).
The strict lower bound on \(\delta\) below and (D56.7) then give
\(\alpha>0\).

Put
\[
L=\frac{K_wK_y}{D},\qquad T=L+\sqrt{\alpha\delta}.
\]
The stationary equations and (D56.6) give
\[
p=\frac{gx}{m-1}-\frac{T}{N},\qquad
\frac{\eta}{\zeta}=\sqrt{\frac{\delta}{\alpha}}.   \tag{D56.9}
\]
Entry positivity therefore implies
\[
p<p_*:=\min\left\{
y^2\sqrt{\frac{\alpha}{\delta}},
w^2\sqrt{\frac{\delta}{\alpha}}\right\}.           \tag{D56.10}
\]
If (D56.8) were nonpositive, then \(p\geq gx/(2m)\);
combining this with (D56.9) gives
\[
p\geq\frac{T}{m(m+1)m!g^{m-2}}.                   \tag{D56.11}
\]

It remains to contradict (D56.10)--(D56.11). By interchanging \(y,w\),
assume \(w\geq y\), normalize \(y=1\), and put
\(t=g/y,\ q=w/y\). We prove, uniformly for \(m\geq3\),
\[
\delta\geq s_m:=m(m+1)m!t^{m-2},\qquad
q^2\delta\geq\alpha.                               \tag{D56.12}
\]
Set
\[
\begin{gathered}
U=(2q+mt)^m,\quad V=(2+mt)^m,\quad D=U+V-m!t^m,\\
A_m=M_wD-K_w^2,\qquad \Delta_m=M_yD-K_y^2.
\end{gathered}
\]
Only \(D\) depends on \(q\) in
\(\delta=M_y-K_y^2/D\), and \(D\) is increasing, so \(\delta\) is
increasing. Put \(b=m!t^{m-2}\). At \(q=1\), set \(P=(2+mt)^m\).
Then \(D=2P-bt^2,\ K_y=m(P-bt)\), and
\(M_y=m^2P-m(m-1)b\). Consequently
\[
\begin{split}
\frac{\Delta_m-s_mD}{m^2}
&=(P-2b)(2P-bt^2)-(P-bt)^2\\
&=P^2-b(t^2-2t+4)P+b^2t^2,
\end{split}
\]
which is equivalently
\[
\Delta_m-s_mD
=m^2b^2\{z^2-((t-1)^2+3)z+t^2\},\qquad
z=\frac{(2+mt)^m}{b}.                              \tag{D56.13}
\]
The \(t^m\) and \(t^{m-2}\) binomial terms imply
\[
z>t^2+\frac{2m^{m-2}}{(m-2)!}\geq t^2+6
>(t-1)^2+3.
\]
The middle quantity is \(6\) at \(m=3\); its successive ratio is
\(((m+1)/m)^{m-2}(m+1)/(m-1)>1\). Hence
the brace in (D56.13) is positive, proving the first inequality.

For the second inequality put \(c=m!t^{m-1}\). The lower terms in
\(q^2M_y-M_w\) cancel, leaving \(m^2(q^2V-U)\). Expanding
\((q^2M_y-M_w)D+K_w^2-q^2K_y^2\), with
\(K_w=m(U-cq), K_y=m(V-c), D=U+V-ct\), and cancelling the square
terms gives
\[
\frac{q^2\Delta_m-A_m}{m^2}
=UV(q^2-1)+m!t^{m-1}\{U(t-2q)+q^2V(2-t)\}.         \tag{D56.14}
\]
Put \(q=1+v,\ W=2+mt\), and
\[
U=\sum_{k=0}^m u_kv^k,\qquad
u_k=\binom mk2^kW^{m-k},
\]
with \(u_k=0\) outside its displayed range. The coefficient \(C_k\) of
\(v^k\) on the right of (D56.14) is
\[
\begin{split}
C_k={}&W^m(2u_{k-1}+u_{k-2})+m!t^{m-1}H_k,\\
H_k={}&(t-2)u_k-2u_{k-1}\\
&+(2-t)W^m(\mathbf1_{k=0}+2\mathbf1_{k=1}+\mathbf1_{k=2}).
\end{split}                                        \tag{D56.15}
\]
Here \(C_0=0\). For \(k=1\),
\[
C_1=2W^{m-1}\!\left[
W^{m+1}-m!t^{m-1}
\{mt^2-2(m-1)t+2(m-1)\}\right].
\]
Only degrees \(m-1,m+1\) in the bracket can be negative. Their positive
coefficients \(2(m+1)m^m,m^{m+1}\) exceed
\(2(m-1)m!,m\,m!\), respectively.

For \(k=2\),
\[
\begin{split}
C_2={}&W^{m-2}\{W^{m+1}(4m+W)+m!t^{m-1}R_m(t)\},\\
R_m(t)={}&-4(m-1)(m+2)-2(m-1)(m-2)t\\
&+2m(m-2)t^2-m^2t^3.
\end{split}
\]
Only degrees \(m-1,m,m+2\) can be negative. The first two are dominated
by \(4mW^{m+1}\), since
\[
8m(m+1)m^m>
4(m-1)(m+2)m!,\quad
8m(m+1)m^m>2(m-1)(m-2)m!,
\]
and \(m^{m+2}>m^2m!\) handles the last.

For \(3\leq k\leq m\),
\[
H_k=-\frac{2^k(m+1)}{m-k+1}\binom mk
W^{m-k}\{(k-1)t+2\}.
\]
After factoring \(2^{k-2}W^{m-k}\), the remaining polynomial is
\[
\begin{split}
4\binom m{k-1}W^{m+1}+\binom m{k-2}W^{m+2}
{}-\frac{4m!(m+1)}{m-k+1}\binom mk
t^{m-1}\{(k-1)t+2\}.
\end{split}
\]
Only degrees \(m-1,m\) can be negative. The coefficient of the first
positive term at both degrees is
\(8\binom m{k-1}(m+1)m^m\). Using
\(\binom mk/(m-k+1)=\binom m{k-1}/k\), the negative coefficients are
\[
\frac{8m!(m+1)}k\binom m{k-1},\qquad
\frac{4m!(m+1)(k-1)}k\binom m{k-1},
\]
both smaller because \(m^m>m!\). Finally,
\[
C_{m+1}=2^{m-1}\{W^m(4+mW)-4m!t^{m-1}\}>_{\rm coeff}0,\quad
C_{m+2}=2^mW^m>_{\rm coeff}0;
\]
at the only threatened degree of \(C_{m+1}\),
\(4W^m\) contributes \(8m^m>4m!\). Thus every nonconstant coefficient
in (D56.14) is positive, proving the second inequality.

Appendix~\ref{app:uniform-bookkeeping} records a direct symbolic-expansion
scheme for the permanent, odd expansion, cofactor sum, and coefficient
inequalities.  It is included only as transparent bookkeeping for the
lowest multiplicities \(m=3,4\); the all-\(m\) proof is the coefficient
argument above.

The second inequality in (D56.12) selects the first term in (D56.10).
Restoring the normalization gives
\(\delta\geq m(m+1)m!g^{m-2}y^2\). The first inequality, together with
\(L>0\), therefore gives
\[
T\geq\sqrt{\alpha\delta}
=\delta\sqrt{\alpha/\delta}
\geq m(m+1)m!g^{m-2}p_*.
\]
This contradicts (D56.10)--(D56.11). Thus (D56.8) is always positive, so
(D56.1) cannot be a maximizer. The stratum bridge excludes every boundary
maximizer, and Hwang's positive-support theorem identifies the unique
maximizer as \(U_n\).
\end{proof}

\section{Completion of the proof}\label{sec:completion}

\begin{proof}[Proof of Theorems~\ref{thm:boundary-exclusion} and
\ref{thm:dittert}]
Proposition~\ref{prop:dittert-n4-supports} excludes boundary maximizers in
order four, and Proposition~\ref{prop:dittert-nge5} excludes them in every
order \(n\geq5\).  This proves Theorem~\ref{thm:boundary-exclusion}.
Theorem~\ref{thm:hwang-positive} then identifies the unique maximizer as
\(U_n\) for every \(n\geq4\).  Sinkhorn's theorem gives \(n=2\), and
Hwang's theorem gives \(n=3\) \cite{Sinkhorn,Hwang87}.  Hence
Conjecture~\ref{conj:dittert} holds in every dimension, proving
Theorem~\ref{thm:dittert}.
\end{proof}

\appendix

\section{Exact rational verification of the spectral alternatives}
\label{app:spectral-certificate}

This appendix specifies the finite arithmetic used in
Theorem~\ref{thm:dittert-7-13}.  It is included to make the argument
reproducible from this paper alone.  The Sturm-sequence and exact
root-isolation procedures are standard; see \cite{BasuPollackRoy}.  Every input is one of the rational
numbers printed in the proof, and every comparison is performed after
cross-multiplication over \(\mathbb Z\).

For a polynomial \(p(x)=\sum_{i=0}^d p_ix^i\) and a rational interval
\(I=[a,b]\), define interval Horner evaluation recursively by
\[
 H_d=[p_d,p_d],\qquad H_i=I H_{i+1}+[p_i,p_i],
\]
where the product of two intervals is the interval spanned by the four
endpoint products.  Thus \(p(I)\subseteq H_0\), with no rounding.  For a
rational symmetric matrix \(M\), exact \(LDL^T\) elimination uses
\[
 d_i=M_{ii}-\sum_{k<i}l_{ik}^2d_k,\qquad
 l_{ji}=\frac{M_{ji}-\sum_{k<i}l_{jk}l_{ik}d_k}{d_i}.
\]
All pivots \(d_i>0\) certify positive definiteness.

The finite procedure is the following.

\begin{enumerate}
\item Evaluate every displayed AM--GM marginal inequality and every
  block-permanent inequality as a rational number.
\item For every block order \(m\), enumerate the pairs
  \(0\leq i,j\leq m\) with \(i+j>m\) and check the displayed Hall-loss
  inequality.
\item Form the weighted path matrix from the displayed \(H_\ell\) and
  marginal bounds and apply the exact \(LDL^T\) recursion above.
\item For \(n=8\), subdivide
  \([0,491/10^4]\) into \(16\) equal rational intervals and apply interval
  Horner evaluation to each of the four degree-eight polynomials in
  (14).
\item For \(n=7\), subdivide
  \([0,783/10^4]\) into \(4096\) equal rational intervals.  On each
  interval and for each \(\ell=1,2,3\), choose the least
  \(B\in10^{-8}\mathbb Z\) for which the first polynomial in (18) has
  positive interval-Horner lower endpoint; then check the second
  polynomial in (18) with that same \(B\).  Finally apply interval Horner
  evaluation to the cubic spectral comparison following (18).
\end{enumerate}

The following table gives strict rational lower bounds for the smallest
output in every finite family of comparisons.  A dash means that the
corresponding dimension uses no check of that type.  The headings ``local
1'' and ``local 2'' refer, in order, to the additional Hall/leakage or
Hall/polynomial tests described in the proof.

\begin{exacttable}
\[
\begin{array}{c|c|c|c|c|c|c}
n&\text{marginal or }a_n&\text{local 1}&\text{local 2}
 &\text{block}&\text{path pivot}&\text{spectral}\\ \hline
7&3/10^5&1/10^3&1/10^{12}&1/10^6&1/1250&1/10^3\\
8&1/200000&1/10^3&1/20000&-&1/500&3/10^4\\
9&1/500000&3/10^6&-&1/125000&1/5000&1/25000\\
10&19/10^9&-&-&9/10^7&-&1/10^4\\
11&1/(4\cdot10^7)&-&-&1/250000&-&3/50000\\
12&1/10^6&-&-&1/10^6&-&1/2500\\
13&9/10^7&-&-&1/10^6&-&1/2000
\end{array}
\tag{A.1}
\]
\end{exacttable}
Each entry is a deliberately shortened lower bound, not a decimal
approximation to the actual margin.

\section{The prescribed-two-zero Sturm calculation}
\label{app:sturm-certificate}

Let \(P_n\) be the polynomial displayed before
Lemma~\ref{lem:dittert-certificate}.  Starting with
\[
 S_0=P_n',\qquad S_1=P_n'',\qquad
 S_{k+1}=-\operatorname{rem}(S_{k-1},S_k),
\]
delete zero terminal terms to obtain the Sturm sequence.  At a rational
number \(x\), let \(V(x)\) be the number of sign changes in
\((S_0(x),S_1(x),\ldots)\), after zero entries are omitted.  The number of
roots of \(P_n'\) in \((a,b)\) is \(V(a)-V(b)\).  Repeated rational
bisection isolates every such root, after which the interval-Horner rule
of Appendix~\ref{app:spectral-certificate} bounds \(P_n\) there.

The complete output needed in Lemma~\ref{lem:dittert-certificate} is
\begin{exacttable}
\begin{equation}
\resizebox{\textwidth}{!}{$
\begin{array}{c|c|c|c|c|c}
n&\deg P_n&
\begin{gathered}\text{critical-point}\\[-2pt]\text{isolating interval}\end{gathered}
&\#\text{ roots}&
\begin{gathered}q_n\text{ lower}\\[-2pt]\text{bound}\end{gathered}&
\begin{gathered}f_n(q_n)\text{ lower}\\[-2pt]\text{bound}\end{gathered}\\ \hline
14&14&
[7647323/10^8,\;1911831/25000000]&1&
15781/(2\cdot10^9)&1/500000000\\
15&15&
[3553391/50000000,\;7106783/10^8]&1&
3001/10^9&7/10^9
\end{array}$}
\tag{B.1}
\end{equation}
\end{exacttable}
where
\[
 f_n(q)=q-\gamma_n-\frac{n^3q^2}{4(1-\gamma_n)}.
\]
Exact endpoint evaluation gives
\begin{exacttable}
\begin{equation}
\resizebox{\textwidth}{!}{$
\begin{array}{c|c|c}
n&P_n(0)&P_n(1/(n-2))\\ \hline
14&1925/35831808&
206787109375/25999348907114496\\
15&479001600/23298085122481&
1639976960656103731200/542800770374370512771595361
\end{array}$}
\tag{B.2}
\end{equation}
\end{exacttable}
and both endpoint values exceed the corresponding lower bound in
(B.1).  Interval Horner evaluation on the two isolating intervals gives
the same strict lower bounds.  Hence the table covers every possible
global minimum.

Finally,
\[
 f_n'(q)=1-\frac{n^3q}{2(1-\gamma_n)}.
\]
Substituting the endpoint upper bound \(q\leq P_n(0)\) makes this positive
for \(n=14,15\), again by integer cross-multiplication.  Therefore the
lower bounds for \(f_n(q_n)\) in (B.1) follow from the lower bounds for
\(q_n\).

For the stated limitation at \(n=13\), take
\[
 \delta=\frac{\gamma_{13}}{91}
 =\frac{68428800}{302875106592253},\quad
 x_0=\frac{82762562}{10^9},\quad
 t_-=\frac{856898537}{5\cdot10^{11}}.
\]
Then \(t_-^2\leq13\delta/(1-\delta)\), and exact evaluation gives
\[
 P_{13}(x_0)(1-t_-)^{13}-(\gamma_{13}-\delta)
 <-\frac9{10^8}.
\tag{B.3}
\]
Thus the scalar comparison really does fail in dimension \(13\).

\section{Capacity-route countermodels}
\label{app:capacity-controls}

The diagnostic in the dimension-four discussion is also completely
rational.  For the data \((s,d,x)\) in the table, put
\[
\begin{split}
\rho&=1-\left(1-\frac ds\right)^s
          \left(1+\frac d{n-s}\right)^{n-s},\\
\mu&=1-\frac{(2-\gamma_n)d}{1-\rho},\qquad
q^+=P_n(x),\\
M&=q^+(1-\rho)\mu^n-(\gamma_n-2\rho).
\end{split}
\]
The following strict bounds follow by direct cross-multiplication.
\begin{exacttable}
\[
\begin{array}{c|c|c|c|c|c|c}
n&s&d&x&\rho\text{ upper}&q^+\text{ upper}&M\text{ upper}\\ \hline
4&2&147/1000&3793/12500&11/1000&103/1000&-45/1000\\
5&2&101/1000&4723/20000&43/10000&41/1000&-16/1000\\
6&3&69/1000&963/5000&16/10000&17/1000&-5/1000
\end{array}
\tag{C.1}
\]
\end{exacttable}
Since \(M<0\) in every row, even the strengthened scalar capacity route
does not reach the required permanent lower bound.

\section{Complete order-four support enumeration}
\label{app:support-enumeration}

Encode a zero set by a \(16\)-bit mask.  A mask is admissible if

\begin{enumerate}
\item it contains two entries in distinct rows and distinct columns; and
\item for no nonempty proper \(I\subset\{1,2,3,4\}\) does it contain the
  entire rectangle \(I\times J\) with \(|J|=4-|I|\).
\end{enumerate}

The second condition is precisely the zero-pattern test for full
indecomposability.  For every admissible mask, take the lexicographically
least image under independent row and column permutations and transpose.
The following pseudocode is an exhaustive specification:
\[
\begin{array}{l}
\mathcal C\leftarrow\varnothing;\\
\textbf{for }z=1,\ldots,2^{16}-1:\\
\quad\textbf{if }z\text{ satisfies (1) and (2):}\\
\qquad c\leftarrow\min\{\tau(\sigma_r z\sigma_c):
 \sigma_r,\sigma_c\in S_4,\ \tau\in\{1,\text{transpose}\}\};\\
\qquad\mathcal C\leftarrow\mathcal C\cup\{c\}.
\end{array}
\tag{D.1}
\]
Writing \(ij\) for the zero entry \((i,j)\), the resulting \(23\)
representatives are:
\begin{exacttable}
\[
\begin{array}{c|c|l@{\qquad}c|c|l}
\#&|Z|&Z&\#&|Z|&Z\\ \hline
1&2&12,21&13&5&12,14,22,23,31\\
2&3&11,12,21&14&5&13,14,22,31,41\\
3&3&12,13,21&15&5&11,14,23,32,41\\
4&3&13,22,31&16&6&12,13,21,23,31,32\\
5&4&11,13,21,22&17&6&12,14,21,23,31,32\\
6&4&13,14,21,22&18&6&12,14,22,23,31,41\\
7&4&12,13,21,31&19&6&12,14,21,23,32,41\\
8&4&11,13,22,31&20&6&13,14,21,23,32,41\\
9&4&13,14,22,31&21&7&13,14,22,23,31,32,41\\
10&4&14,23,32,41&22&7&13,14,22,24,32,33,41\\
11&5&12,13,21,22,31&23&8&13,14,22,24,31,33,41,42\\
12&5&13,14,21,22,31&&&
\end{array}
\tag{D.2}
\]
\end{exacttable}
Counting by \(|Z|=2,\ldots,8\) gives
\((1,3,6,5,5,2,1)\), as asserted.

\section{The order-four Bernstein certificate}
\label{app:bernstein-certificate}

We give the exact arithmetic suppressed in the transition from (D4.5) to
(D4.7).  After \(y=1\), \(t=g/y\), and \(q=w/y\), write
\[
\begin{gathered}
D=3t^2+4tq+4t+2q^2+2,\\
K_w=8t^2+12tq+8q^2,\qquad K_y=8t^2+12t+8,\\
M_w=8t^2+16tq+6q^2,\qquad M_y=8t^2+16t+6,\\
A=2M_wD-\tfrac12K_w^2,\qquad
\Delta=2M_yD-\tfrac12K_y^2.
\end{gathered}
\]
Thus \(\alpha=A/D\), \(\delta=\Delta/D\), and
\(L=K_wK_y/(2D)\).

First, the numerator
\[
S(t,q)=q^2\Delta-A
\]
satisfies, on writing \(q=1+v\),
\[
\begin{split}
S(t,1+v)={}&
(64v+160v^2+128v^3+32v^4)\\
&+t(224v+448v^2+288v^3+64v^4)\\
&+t^2(384v+576v^2+256v^3+32v^4)\\
&+t^3(256v+256v^2+64v^3)
+t^4(32v+16v^2).
\end{split}
\tag{E.1}
\]
Hence \(q^2\delta\geq\alpha\) for \(q\geq1\), so the first term is the
minimum in (D4.5).

Next put
\[
\begin{split}
j(t,q)&=(\Delta-12D)/8\\
&=-4-4t+8tq^2+3t^2+16t^2q+4t^2q^2
  +8t^3+8t^3q+2t^4.
\end{split}
\]
It is increasing in \(q\geq1\), while
\[
j(t,1)=2t^4+16t^3+23t^2+4t-4,\qquad
j(5/16,1)=113/32768>0.
\]
Thus \(\delta<12\), equivalently \(j(t,q)<0\), forces \(t<5/16\).
Finally let \(i(t,q)=-A/4\).  Its derivative with respect to \(t\) has
strictly negative coefficients, and
\[
i(5/16,3+v)=
\frac{2767}{16384}+\frac{27225}{256}v
+\frac{21897}{256}v^2+\frac{91}{4}v^3+2v^4>0.
\tag{E.2}
\]
Since \(\alpha>0\) means \(i<0\), (E.2) forces \(q<3\).

Define the polynomial \(h\) by
\[
\frac14K_w^2K_y^2\Delta-(12D-\Delta)^2A
=1024D\,h(t,q).
\tag{E.3}
\]
Because \(D>0\), inequality (D4.7) is equivalent to \(h\geq0\).
On the rectangle \(t=5u/16,\ q=1+2v\), write
\[
h(5u/16,1+2v)
=\sum_{r=0}^{10}\sum_{s=0}^{6}
b_{rs}\binom{10}{r}u^r(1-u)^{10-r}
{}\binom{6}{s}v^s(1-v)^{6-s}.
\tag{E.4}
\]
For compactness, the table gives the exact integers
\[
\widehat b_{rs}
=16^{10}\binom{10}{r}\binom{6}{s}b_{rs}.
\]
\begin{exacttable}
\begin{equation}
\resizebox{\textwidth}{!}{$
\begin{array}{c|rrrr}
r\backslash s&0&1&2&3\\ \hline
0&0&52776558133248&527765581332480&2005509209063424\\
1&10995116277760&755914244096000&6797730638725120&24744509183098880\\
2&131597797949440&4871953202544640&39783766666772480&138818668766167040\\
3&684338614108160&18402251773050880&138255740824453120&463593245528555520\\
4&2037861989744640&44889870077788160&314269358547271680&1015984788235878400\\
5&3840421773967360&73771295038242816&486571381652193280&1521712544442482688\\
6&4759162783989760&82707597515161600&518563451833221120&1573843691725127680\\
7&3883644615229440&62501158991626240&375187130506772480&1108144168546467840\\
8&2014444474940160&30493375333793280&176248115411877120&507813903405982720\\
9&603126347120640&8680816622338560&48525937286791680&136667246148121600\\
10&79459459141860&1095864216308568&5945463454398300&16395153287757264
\end{array}$}
\tag{E.5a}
\end{equation}
\[
\begin{array}{c|rrr}
r\backslash s&4&5&6\\ \hline
0&3588805953060864&3008263813595136&949978046398464\\
1&43238294762291200&35720384007372800&11173786917273600\\
2&237112774906347520&193664284444590080&60436718205009920\\
3&775102262820208640&627284432293724160&196151556377149440\\
4&1665176387746529280&1337219171766763520&419932056877793280\\
5&2448320986913701888&1952660699807219712&616306674696716288\\
6&2488951360659456000&1972460755271680000&625677506791014400\\
7&1724493900574392320&1358340519418593280&432834714309918720\\
8&778409822206915840&609511543169768960&194986892197712640\\
9&206526568418452480&160778523305208320&51607187621460480\\
10&24442862602674524&18920540600095576&6090767843652324
\end{array}
\tag{E.5b}
\]
\end{exacttable}
All \(77\) entries in (E.5a)--(E.5b) are nonnegative and only
\(\widehat b_{00}\) vanishes.  Since the tensor-product Bernstein basis is
nonnegative on \([0,1]^2\), (E.4)--(E.5b) prove \(h\geq0\), with strict
positivity in the case used in the proof.

\section{Bookkeeping for the uniform multiplicity identities}
\label{app:uniform-bookkeeping}

The universal argument following (D56.12) is symbolic and does not depend
on finite testing.  For transparency, its lowest-multiplicity expansion can
be reproduced by the following finite algebraic recipe:

\begin{enumerate}
\item Form the seven-variable block matrix (D56.1), retaining the last
  row and column multiplicity as the formal integer \(m\).
\item Enumerate the three possible assignment types of the first two rows:
  both use the two exceptional columns, exactly one does, or neither does.
  Their multiplicities give the three terms of (D56.3).
\item Substitute the odd coordinates, expand the two marginal products and
  (D56.3), and collect in \(\xi,\eta,\zeta\).  This gives (D56.4).
\item Form \(D,K_w,K_y,M_w,M_y,N\), eliminate \(\xi\), and collect the
  numerator of each assertion in (D56.12) first in \(q-1\) and then in
  \(t\).
\end{enumerate}

For the two lowest multiplicities, the exact coefficient bookkeeping is:
\begin{exacttable}
\[
\begin{array}{c|c|c|c|c}
m&n&
\begin{gathered}\#\text{ lower-numerator}\\[-2pt]\text{terms for }\delta\end{gathered}
&\begin{gathered}\text{least positive}\\[-2pt]\text{coefficient}\end{gathered}
&\begin{gathered}\#\text{ selector terms},\\[-2pt]\text{least coefficient}\end{gathered}\\ \hline
3&5&7&576&(29,576)\\
4&6&9&4096&(44,4096)
\end{array}
\tag{F.1}
\]
\end{exacttable}
These rows are consistency checks on the displayed identities, not a
substitute for the coefficientwise proof for arbitrary \(m\).

\section*{Statement on AI-assisted discovery and human verification}
The TARS agent system assisted with exploratory derivations with an autonomous mathematical search. X. Dai completed proof reconstruction, supplemented omitted details, literature review and manuscript writing, and manually verified all mathematical results.

\end{document}